\documentclass[10pt]{elsarticle}

\usepackage[english]{babel} 
\usepackage{ae,aecompl} 
\usepackage{amsmath,amssymb,amsfonts,latexsym} 
\usepackage{amsthm} 
\usepackage{bm} 
\usepackage{enumerate}
\usepackage{geometry,setspace}  
\usepackage{color}

\usepackage{url}
\usepackage[bookmarks=true,colorlinks=true,linkcolor=blue, citecolor=green,urlcolor=red,plainpages=false,pdftex]{hyperref}

\theoremstyle{plain} 
 \newtheorem{theorem}{Theorem}[section]   
 \newtheorem{proposition}[theorem]{Proposition}

\theoremstyle{definition} 
 \newtheorem{remark}[theorem]{Remark}
 \newtheorem{definition}[theorem]{Definition}
\theoremstyle{remark} 

\numberwithin{equation}{section}  
\allowdisplaybreaks 

\newcommand{\R}{\ensuremath{\mathbb{R}}} 
\newcommand{\C}{\ensuremath{\mathbb{C}}} 
\newcommand{\U}{\ensuremath{\mathcal{U}}} 
\newcommand{\V}{\ensuremath{\mathcal{V}}} 
\newcommand{\W}{\ensuremath{\mathcal{W}}} 
\newcommand{\qbinom}[3]{\ensuremath{\left[\begin{array}{c} #1 \\ #2 \\ \end{array}\right]_{#3}}} 

\begin{document}


\begin{frontmatter}
\title{On linearly related sequences of\\ difference derivatives of discrete orthogonal polynomials}

 \author[1]{R. \'{A}lvarez-Nodarse} \ead{ran@us.es}
 \author[2]{J. Petronilho} \ead{josep@mat.uc.pt}
 \author[3]{N. C. Pinz\'on-Cort\'es} \ead{pinzon.natalia@javeriana.edu.co}
 \author[4]{R. Sevinik-Ad\i g{\"{u}}zel} \ead{sevinikrezan@gmail.com}

  \address[1]{Universidad de Sevilla, Departamento de An\'alisis Matem\'atico, IMUS.\\ Apdo. 1160, E-41080, Sevilla, Spain.}
  \address[2]{University of Coimbra, Department of Mathematics, CMUC.\\ EC Santa Cruz, 3001-501 Coimbra, Portugal.}
  \address[3]{Pontificia Universidad Javeriana. Departamento de Matemáticas. Carrera 7, 40 - 62, Bogotá, Colombia}
  \address[4]{Sel\c{c}uk University, Department of Mathematics, Faculty of Science.\\ 42075, Konya, Turkey.}


\begin{abstract}
Let $\nu$ be either $\omega\in\C\setminus\{0\}$ or
$q\in\C\setminus\{0,1\}$, and let $D_\nu$ be the corresponding
difference operator defined in the usual way either by $D_\omega p(x) = \frac{p(x+\omega)-p(x)}{\omega}$ or
$D_q p(x) = \frac{p(qx)-p(x)}{(q-1)x}$. Let $\U$ and $\V$ be two moment regular linear functionals
and let $\{P_n(x)\}_{n\geq0}$ and $\{Q_n(x)\}_{n\geq0}$ be their corresponding orthogonal
polynomial sequences (OPS). We discuss an inverse problem in the theory of
discrete orthogonal polynomials involving the two OPS $\{P_n(x)\}_{n\geq0}$ and $\{Q_n(x)\}_{n\geq0}$
assuming that their difference derivatives $D_\nu$
of higher orders $m$ and $k$ (resp.) are connected by a linear algebraic
structure relation such as
\begin{equation*}
   \sum_{i=0}^M a_{i,n} D_\nu^m P_{n+m-i}(x) = \sum_{i=0}^N b_{i,n} D_\nu^k Q_{n+k-i}(x), \quad n\geq 0,
  \end{equation*}
where $M,N,m,k\in\mathbb{N}\cup\{0\}$, $a_{M,n}\neq0$ for $n\geq M$, $b_{N,n}\neq0$ for $n\geq N$,
and $a_{i,n}=b_{i,n}=0$ for $i>n$.
Under certain conditions, we prove that $\U$ and $\V$ are related by a rational factor (in the $\nu-$distributional sense).
Moreover, when $m\neq k$ then both $\U$ and $\V$ are $D_\nu$-semiclassical functionals. This leads us to the concept of $(M,N)$-$D_\nu$-coherent pair of order $(m,k)$ extending to the discrete case several previous works.
As an application we consider the OPS with respect to the following Sobolev-type inner product
  \begin{equation*}
   \left\langle p(x),r(x) \right\rangle_{\lambda,\nu} = \left\langle \U, p(x)r(x) \right\rangle  + \lambda \left\langle \V, (D^m_\nu p)(x)(D^m_\nu r)(x) \right\rangle, \quad \lambda>0,
  \end{equation*}
assuming that $\U$ and $\V$ (which, eventually, may be represented by discrete measures supported either on a uniform lattice if $\nu=\omega$, or on a $q$-lattice if $\nu=q$) constitute a $(M,N)$-$D_\nu$-coherent pair of order $m$ (that is, an $(M,N)$-$D_\nu$-coherent pair of order $(m,0)$), $m\in\mathbb{N}$ being fixed.
\end{abstract}


\end{frontmatter}


\section{Introduction}
\label{introduction}

An interesting problem in the theory of orthogonal polynomials is the
one associated with linearly related sequences of derivatives of
two sequences of polynomials (see e.g.
\cite{Jesus_Marcellan_Petronilho_PinzonCortes201X,Jesus_Petronilho_2008,Jesus_Petronilho_2013,Petronilho_2006}
and references therein). To be more precise assume that $\U$ and $\V$ are two
regular functionals and suppose that their corresponding orthogonal polynomial sequences (OPS)
 $\{P_n(x)\}_{n\geq0}$  and $\{Q_n(x)\}_{n\geq0}$ are connected by the following linear structure relation
\begin{equation}\label{coh-intro}
\sum_{i=0}^M a_{i,n} D^m P_{n+m-i}(x) = \sum_{i=0}^N b_{i,n} D^k Q_{n+k-i}(x), \quad n\geq 0,
\end{equation}
where $M,N,m,k\in\mathbb{N}\cup\{0\}$, $a_{i,n}$ and $b_{i,n}$ are complex parameters with $a_{M,n}\neq0$
for $n\geq M$, $b_{N,n}\neq0$ for $n\geq N$, and
$a_{i,n}=b_{i,n}=0$ for $i>n$, and $D^j$ is the (continuous) derivative operator of order $j$.
The pair $(\U,\V)$ such that the above  relation \eqref{coh-intro} holds is said to be a
$(M,N)$-coherent pair of order $(m,k)$.

Historically, the notion of coherent pair --i.e. $(1,0)$-coherent pair of order $(1,0)$,
according with the terminology above-- arose in the framework of the theory of Sobolev orthogonal polynomials and it was introduced by A. Iserles, P. E. Koch, S. P. N{\o}rsett and J. M. Sanz-Serna in the very influent work \cite{Iserles_Koch_Norsett_Sanz-Serna_1991}. Subsequent extensions of this notion have been widely introduced and studied in recent decades. For a review of the work done on the subject (including an historical perspective) see e.g. the introductory sections in the papers \cite{Jesus_Marcellan_Petronilho_PinzonCortes201X,Jesus_Petronilho_2013,Marcellan_PinzonCortes_2012}.
For instance, it is known that for a pair of positive definite linear functionals $(\U,\V)$, coherence of order $m$ ---i.e., of order $(m,0)$--- is a necessary and sufficient condition for the existence of an algebraic structure relation between the OPS with respect to $\U$ and the Sobolev OPS with respect to an appropriate inner product defined in terms of the measures $\mu_0$ and $\mu_1$ associated with $\U$ and $\V$ (resp.), such as
$$
 \left\langle p(x),q(x)\right\rangle_{\lambda}=\int_\mathbb{R} p(x)q(x) d\mu_0 + \lambda \int_\mathbb{R} p^{(m)}(x)q^{(m)}(x) d\mu_1, \lambda>0, m\in\mathbb{N}
$$
for every polynomials $p,q\in\mathbb{P}$. Indeed, this fact has been firstly remarked (and proved) in \cite{Iserles_Koch_Norsett_Sanz-Serna_1991} for ordinary coherent pairs (i.e., $(M,N,m,k)=(1,0,1,0)$), and stated in \cite{kcoerencia,mar-fil1998} for $N=0$ and $m=1$ (being $M$ arbitrary and $k=0$). The ideas presented in \cite{kcoerencia,mar-fil1998} have led to the statement of the mentioned structure relation for arbitrary $(M,N,m)$ (see \cite{Jesus_Petronilho_2013} for the case $m=1$ and \cite{Jesus_Marcellan_Petronilho_PinzonCortes201X} for arbitrary $m\geq1$). On the other hand, when $(\U,\V)$ is a $(M,N)$-coherent pair of order $(m,k)$ of regular linear functionals, it is known that these linear functionals are related by an expression of rational type in the distributional sense and, moreover, they are semiclassical when $m\neq k$ (see \cite{Petronilho_2006} for the case $m=k=0$,  \cite{Jesus_Petronilho_2008} for the cases $m=k$ and $m=k+1$, and \cite{Jesus_Marcellan_Petronilho_PinzonCortes201X} for arbitrary $m>k+1$).

The concept of coherent pair was extended to the OPS of a discrete variable by I. Area, E. Godoy, and F.
Marcell\'{a}n \cite{Area_Godoy_Marcellan_2000, Area_Godoy_Marcellan_2002, Area_Godoy_Marcellan_2003}, and also by F. Marcell\'an and N. Pinz\'on-Cort\'ez \cite{Marcellan_PinzonCortes_2012_Dw,Marcellan_PinzonCortes_2012_Dq}. Here we generalize this concept as follows. Let $\U$ and $\V$ be two regular linear functionals and let $\{P_n(x)\}_{n\geq0}$ and $\{Q_n(x)\}_{n\geq0}$ be their respective sequences of monic orthogonal polynomials (SMOP). $(\U,\V)$ is a \emph{$(M,N)$-$D_\nu$-coherent pair of order $(m,k)$}, for either $\nu=\omega\in\C\setminus\{0\}$ or $\nu=q\in\C\setminus\{0,1\}$, if the algebraic relation
\begin{equation*}
 \sum_{i=0}^M a_{i,n} D_\nu^m P_{n+m-i}(x) = \sum_{i=0}^N b_{i,n} D_\nu^k Q_{n+k-i}(x), \quad n\geq 0,
\end{equation*}
holds, where $M,N,m,k\in\mathbb{N}\cup\{0\}$, $\{a_{i,n}\}_{n\geq0}, \{b_{j,n}\}_{n\geq0}\subset\C$ for $0\leq i\leq M$ and $0\leq j\leq N$, $a_{M,n}\neq0$ for $n\geq M$, $b_{N,n}\neq0$ for $n\geq N$, $a_{i,n}=b_{i,n}=0$ for $i>n$, and
\begin{equation*}
 D_\omega p(x) = \frac{p(x+\omega)-p(x)}{\omega}, \quad D_q p(x) = \frac{p(qx)-p(x)}{(q-1)x}, \quad p\in\mathbb{P}.
\end{equation*}

Marcell\'{a}n and N. C. Pinz\'{o}n-Cort\'{e}s (\cite{Marcellan_PinzonCortes_2012_Dw} for $\nu=\omega$, \cite{Marcellan_PinzonCortes_2012_Dq} for $\nu=q$) showed that if $(\U,\V)$ is a $(1,1)$-$D_\nu$-coherent pair then they are $D_\nu$-semiclassical linear functionals (one of class at most $1$ and the other of class at most $5$) and they are related by $\sigma(x)\U=\rho(x)\V$, with $\mathrm{deg}(\sigma(x))\leq3, \mathrm{deg}(\rho(x))=1$. Also, they studied the case when $\U$ is $D_\nu$-classical. This is a generalization of the results obtained by I. Area, E. Godoy, and F. Marcell\'{a}n (\cite{Area_Godoy_Marcellan_2000, Area_Godoy_Marcellan_2003} for $\nu=\omega$, \cite{Area_Godoy_Marcellan_2002} for $\nu=q$) for $(1,0)$-$D_\nu$-coherent pairs. They proved that $(1,0)$-$D_\nu$-coherence is a sufficient condition for at least one of the linear functionals to be $D_\nu$-classical and each of them to be a rational modification of the other as above with $\mathrm{deg}(\sigma(x))\leq2$. Besides, they determined all $D_\nu$-coherent pairs of positive definite linear functionals when $\U$ or $\V$ is some specifical $D_\nu$-classical linear functional. Notice that from the study of $D_\nu$-coherent pairs it is possible to recover the properties of coherent pairs in the continuous case taking limits when $\omega\rightarrow0$ and $q\rightarrow1$.

As before, there is an important connection between $D_\nu$-Sobolev orthogonal polynomials and $D_\nu$-coherent pairs. In fact, we can consider the Sobolev inner product
\begin{equation}  \label{eq-sob-intro}
 \left\langle p(x),r(x) \right\rangle_{\lambda,\nu} = \left\langle \U, p(x)r(x) \right\rangle  + \lambda \left\langle \V, (D^m_\nu p)(x)(D^m_\nu r)(x) \right\rangle, \quad \lambda>0,
\end{equation}
for fixed $m\in\mathbb{N}$, when $\U$ and $\V$ (which will be supported on, either a uniform lattice if $\nu=\omega$, or a $q$-lattice if $\nu=q$) constitute a $(M,N)$-$D_\nu$-coherent pair of order $m$ (i.e., order $(m,0)$). In this way, K. H. Kwon, J. H. Lee and F. Marcell\'an (\cite{Kwon_Lee_Marcellan_2004}) showed that the $(M,0)$-$D_\omega$-coherence of order $1$ condition (for them, \emph{($M+1$-term) generalized $D_\omega$-coherence}) yields the relation
\begin{equation} \label{eq-Dw-intro}
 P_{n+1}(x) + \sum_{j=1}^{M} \frac{(n+1)a_{j,n}}{n-j+1} P_{n-j+1}(x) = S_{n+1}(x;\lambda,\omega) + \sum_{j=1}^{M}c_{j,n,\lambda,\omega}S_{n-j+1}(x;\lambda,\omega),
\end{equation}
for $n\geq M$, where $\{c_{n,\lambda,\omega}\}_{n\geq M}$ are rational functions in $\lambda>0$, $c_{M,n,\lambda,\omega}\neq0$, $a_{M,n}\neq0$, and $\{S_{n}(x;\lambda,\omega)\}_{n\geq0}$ is the SMOP associated with the inner product \eqref{eq-sob-intro} for $m=1$. Conversely, if \eqref{eq-Dw-intro} holds, then $(\U,\V)$ is a $(M,M)$-$D_\omega$-coherent pair. Additionally, they studied $(2,0)$-$D_1$-coherent pairs of order 1 and they concluded that the linear functionals must be $D_1$-semiclassical (of class $\leq6$ for $\U$ and of class $\leq2$ for $\V$), and they are related by a rational factor. Also, they analized the cases when either $\U$ or $\V$ is a $D_1$-classical linear functional.

 The aim of this work is twofold. On one hand we will extend some recent results concerning the
$(M,N)$-coherent pairs of order $(m,k)$ for the derivative operator and, on the other,
we will show that the concept of $(M,N)$-$D_\nu$ coherent pair of order $(m,k)$ for the discrete analogues
of the derivative $D_\nu$ will play an important role in the study of the
Sobolev orthogonal polynomials on linear and $q$-linear lattices similar to the one
that the  $(M,N)$-coherent pair of order $(m,k)$ plays in the theory of Sobolev orthogonal
polynomials \cite{Jesus_Marcellan_Petronilho_PinzonCortes201X}.
More precisely, we prove that the regular
linear functionals associated to an $(M,N)$-$D_\nu$-coherent pair  $(\U,\V)$ are related
by a rational modification (in the sense of the distribution theory) and, moreover,
$\U$ and $\V$ are both $D_\nu$-semiclassical when $m\neq k$ (see Theorem \ref{main-teo} from below).
As an application, we study the sequence of Sobolev OPS with respect to the Sobolev-type inner product
\eqref{eq-sob-intro}, under the assumption that  $(\U,\V)$ is a $(M,N)$-$D_\nu$-coherent pair
of positive definite discrete linear functionals
(see theorems \ref{teo (M,N)-Dnu-coherence order (m,0) impl relacion con Sobolev, Dw y Dq}
and \ref{teo (M,N)-Dnu-coherence order (m,0) impl eq recursiva s_n,nu y c_j,n,lambda,nu, Dw y Dq}).

The structure of this paper is as follows. In Section 2, we state the definitions, results and notation which will be useful in the forthcoming sections. In Section 3, we prove that if a pair of regular linear functionals form a $(M,N)$-$D_\nu$-coherent pair of order $(m,k)$, then they are related by an expression of rational type, and in the case when $m\neq k$, they are $D_\nu$-semiclassical. In Section 4, we show the relationship between $(M,N)$-$D_\nu$-coherent pairs of order $m$ and $D_\nu$-Sobolev orthogonal polynomials (orthogonal with respect to $\langle \cdot, \cdot \rangle_{\lambda,\nu}$ given in \eqref{eq-sob-intro}).


\section{Preliminaries and Notations}

Let $\mathbb{P}$ be the linear space of polynomials with complex coefficients and let $\mathbb{P}^{\prime}$ be its topological dual space which coincides
with its algebraic dual $\mathbb{P}^*$ \cite{Maroni_1985}. For $\U\in\mathbb{P}^{\prime}$ and $n\geq0$, $u_n=\langle\U,x^n\rangle\in\C$ is called the \emph{moment of order} $n$ of $\U$, where $\langle \U, p(x)\rangle\in\C$ denotes the image of $p\in\mathbb{P}$ by $\U\in\mathbb{P}^{\prime}$.
For $\pi\in\mathbb{P}$, $\pi(x)\U\in\mathbb{P}^{\prime}$ is defined by
$$
\langle\pi(x)\U,p(x)\rangle = \langle\U,\pi(x)p(x)\rangle,\qquad
p\in\mathbb{P}.
$$
Also, for a sequence of polynomials $\{p_n(x)\}_{n\geq0}$ with $\deg(p_n(x))=n$, $n\geq0$, we can consider its \emph{dual basis} $\{\mathfrak{p}_n\}_{n\geq0} \subset \mathbb{P}^\prime$ (i.e., $\langle \mathfrak{p}_n,p_m(x)\rangle=\delta_{n,m}$, $m,n\geq0$). In this way, any $\U\in\mathbb{P}^{\prime}$ can be expanded as $\U = \sum_{n\geq0} \left\langle \U, p_n(x) \right\rangle \mathfrak{p}_n$.

In this paper we will work with the following two linear difference operators on
$\mathbb{P}$
\begin{equation*}
\begin{split}
D_\omega: \,\mathbb{P}\mapsto\mathbb{P},\quad   D_\omega p(x) & = \frac{p(x+\omega)-p(x)}{\omega},\quad
\omega\in\C\setminus\{0\}, \\
 D_q: \,\mathbb{P}\mapsto\mathbb{P},\quad D_q p(x) & = \frac{p(qx)-p(x)}{(q-1)x},\qquad q\in\C\setminus\{0,\pm1\}.
 \end{split}
\end{equation*}
Notice that $D_1=\Delta$ and $D_{-1}=\nabla$ are the well-known forward and backward difference
operators, respectively, and $D_q$ is the classical $q$-derivative operator.

From now on, $\nu$ and $\nu^*$ denote either $\omega$ and $-\omega$, or, $q$ and $q^{-1}$, respectively. Then, for $\U\in\mathbb{P}^{\prime}$, the linear functional $D_\nu\U$ is defined by
\begin{equation*}
 \left\langle D_\nu\U, p(x)\right\rangle = -\left\langle \U, D_{\nu^*} p(x)\right\rangle, \quad p\in\mathbb{P}.
\end{equation*}

Notice that when $q\rightarrow1$ and $\omega\rightarrow0$ we recover the standard derivative operator.
Furthermore, when $\omega\rightarrow0$ and $q\rightarrow1$, $(D_{\nu}p)(x)\rightarrow\frac{d}{dx}p(x)$ in $\mathbb{P}$ and $D_{\nu}\U\rightarrow D\U$ in $\mathbb{P}^{\prime}$, where $D\U$ is defined by $\langle D\U, p(x) \rangle = -\left\langle \U, p'(x)\right\rangle, \forall p\in\mathbb{P}$.

For the difference operators $D_\nu$ the following straightforward  properties hold.
\begin{gather}
 D^m_\omega\left[p(x)\U \right] = \sum_{j=0}^m \binom{m}{j} \bigl(D^j_\omega p\bigr)\bigl(x+(m-j)\omega\bigr) \, D^{m-j}_\omega\U, \quad m\geq0, \label{dw-pu} \\
 D^m_q\left[p(x)\U \right] = \sum_{j=0}^m \qbinom{m}{j}{q} q^j\bigl(D^j_q p\bigr)\bigl(q^{m-j}x\bigr) \, D^{m-j}_q\U, \quad m\geq0, \label{q-dw-pu}
\end{gather}
where the \emph{$q$-binomial coefficient} is defined by
$$
\qbinom{n}{j}{q}:=\frac{(q,q)_n}{(q,q)_j(q,q)_{n-j}}, \quad
n\geq j\geq0,
$$
and $(\alpha;q)_n$ denotes the \emph{$q$-Pochhammer symbol}, which is the $q$-analogue of the
\emph{Pochhammer symbol} $(\alpha)_n$, defined by
\begin{gather*}
 (\alpha)_0:=1, \quad (\alpha)_n := \alpha(\alpha+1)\cdots(\alpha+n-1), \,\, n\geq1, \\
 (\alpha;q)_0:=1,\quad (\alpha;q)_n:=(1-\alpha)(1-\alpha q)\cdots(1-\alpha q^{n-1}), \,\, n\geq1.
\end{gather*}

Let $\U\in\mathbb{P}^{\prime}$ and $\{P_n(x)\}_{n\geq0}\subset\mathbb{P}$. $\{P_n(x)\}_{n\geq0}$ is called the \emph{sequence of monic orthogonal polynomials (SMOP)} with respect to $\U$ if $\deg(P_n(x))=n$ and $\langle \U ,P_n(x)P_m(x) \rangle = \xi_n\delta_{n,m}$, $\xi_n\neq0$, $n,m\geq0$. In this case, $\U$ is said to be \emph{regular or quasi-definite}, and $\Upsilon_n=\det\left([u_{i+j}]_{i,j=0}^n\right)\neq0$,
$\forall n\geq 0$. When $\Upsilon_n>0$, $n\geq 0$, $\U$ is called \emph{positive definite}. An important characterization of OPs is given by the \emph{Favard Theorem}: $\{P_n(x)\}_{n\geq 0}$ is the SMOP with respect to $\U$ if and only if there exist $\{\alpha_{n}\}_{n\geq 0},\{\beta_{n}\}_{n\geq 0}\subset\C$, $\beta_{n}\neq 0$, $n\geq 1$, such that the \emph{three-term recurrence relation (TTRR)} $P_{n+1}(x)=(x- \alpha_{n}) P_{n}(x) - \beta_{n}P_{n-1}(x)$, $n\geq 0$, holds, with $P_{0}(x)=1$, $P_{-1}(x)=0$. Moreover, $\U$ is positive definite if and only if $\alpha_{n}\in\R$ and $\beta_{n+1}>0$, for $n\geq 0$. (see e.g. \cite{Chihara_libro_1978}).

Consider $\{\mathfrak{p}_n\}_{n\geq0}$, the dual basis of the SMOP $\{P_n(x)\}_{n\geq0}$, then for fixed $m\geq 0$,
\begin{equation} \label{eq relacion base dual con U, y Dnu*^m e_n,nu = (-1)^m eta_n,nu p_n+m}
 \mathfrak{p}_n = \frac{P_n(x)}{\left\langle \U, P_{n}^2(x) \right\rangle} \U, \quad D_{\nu^*}^m\mathfrak{e}_{n,\nu} = (-1)^m \eta_{n,m,\nu} \mathfrak{p}_{n+m}, \quad n\geq 0,
\end{equation}
where $\{\mathfrak{e}_{n,\nu}\}_{n\geq0}$ is the dual basis of monic polynomials $\{P^{[m,\nu]}_{n}(x)\}_{n\geq0}$ given by
\begin{equation*}
 P^{[m,\nu]}_{n}(x):=\frac{D_\nu^m P_{n+m}(x)}{\eta_{n,m,\nu}}, \quad\text{with}\quad \eta_{n,m,\omega}:=(n+1)_m, \quad \eta_{n,m,q}:=\frac{(q^{n+1};q)_m}{(1-q)^m}.
\end{equation*}

A regular linear functional $\U\in\mathbb{P}^{\prime}$ is called \emph{$D_{\nu}$-semiclassical} linear functional (see e.g. \cite{Marcellan_Salto_1998} for $\nu=\omega$,
\cite{Kheriji_2003} for $\nu=q$) if it is regular and there exist $\sigma,\tau\in\mathbb{P}$, with $\deg(\tau(x))\geq1$, such that
\begin{equation}\label{pearson}
 D_{\nu}(\sigma(x) \U)=\tau(x) \U.
\end{equation}
In this way, the \emph{class} of $\U$ is $s:=\min\max\left\{{\rm deg}\,\sigma-2,{\rm deg}\,\tau-1\right\}\in\mathbb{N}\cup\{0\}$, where the minimum is taken among all pairs of polynomials $(\sigma,\tau)$, with $\deg(\tau(x))\geq1$, satisfying \eqref{pearson}. When $s=0$, $\U$ is called a \emph{$D_{\nu}$-classical} functional. Besides, the corresponding SMOP is said to be $D_{\nu}$-semiclassical of class $s$, or $D_{\nu}$-classical, respectively.

\begin{proposition} \label{U-semi-D-D*}
 The following equivalence hold
 \begin{gather*}
  D_\nu\left[\sigma(x)\U\right] = \tau(x)\U \quad \Longleftrightarrow \quad D_{\nu^*}\left[\widetilde{\sigma}(x)\U\right] = \tau(x)\U,
 \end{gather*}
 where
 $$
 \widetilde{\sigma}(x):=\left\{
 \begin{array}{ccl}
 \sigma(x) + \omega\tau(x) & \mbox{\rm if} & \nu=\omega\, , \\
 q\sigma(x) + (q-1)x\tau(x) & \mbox{\rm if} & \nu=q\, .
 \end{array}
 \right.
 $$
 Thus, $\U$ is $D_\nu$-semiclassical if and only if it is $D_{\nu^*}$-semiclassical.
\end{proposition}
 The proof of this proposition is straightforward and will be omitted.

\begin{proposition} \label{U-semi-lin}
 If the regular linear functionals $\U,\V$ are related by
\begin{equation}\label{rel-lin}
 p(x)\U = r(x)\V,\qquad p,r\in\mathbb{P}\setminus\{0\},
\end{equation}
then, $\U$ is $D_{\nu}$-semiclassical (respectively $D_{\nu^*}$-semiclassical) if and only if $\V$ also is
$D_{\nu}$-semiclassical (respectively $D_{\nu^*}$-semiclassical). Moreover,
if the class of $\U$ is $s$, then the class of $\V$ is at most $s+\mathrm{deg}(p(x))+\mathrm{deg}(r(x))$.
\end{proposition}
\begin{proof}
 Let us suppose that $\U$ is a $D_\omega$-semiclassical linear functional given by \eqref{pearson}, then $\V$ satisfies
 \begin{multline*}
  D_\omega\left[p(x-\omega)\sigma(x)r(x)\V\right] \overset{\eqref{rel-lin}}{=} D_\omega\left[p(x-\omega)p(x)\sigma(x)\U\right] \\
  \overset{\eqref{dw-pu}}{=} p(x)p(x+\omega)D_\omega\left[\sigma(x)\U\right] + \bigl\{ p(x)D_\omega[p(x)] + p(x)D_\omega[p(x-\omega)] \bigr\}\sigma(x)\U \\
  \underset{\eqref{rel-lin}}{\overset{\eqref{pearson}}{=}} \Bigl(p(x+\omega)\tau(x) + D_\omega\bigl[p(x)+p(x-\omega)\bigr]\sigma(x) \Bigr)r(x)\V.
 \end{multline*}

Therefore, $\V$ is also $D_{\omega}$-semiclassical and the class of $\V$ is at most $s+\mathrm{deg}(p(x))+\mathrm{deg}(r(x))$. The $D_{\nu^*}$-semiclassical character of $\V$ follows from Proposition \ref{U-semi-D-D*}.

The proof of the $q$-case is similar but using \eqref{q-dw-pu} instead of \eqref{dw-pu},
and in this case $\V$ satisfies
\begin{equation*}
 D_q\bigl[p(q^{-1}x)\sigma(x)r(x)\V\bigr]  = \Bigl(p(qx)\tau(x) + qD_q\bigl[p(x)+p(q^{-1}x)\bigr]\sigma(x) \Bigr)r(x)\V.
\end{equation*}
\end{proof}

A characterization of $D_\nu$-semiclassical linear functionals is the following
\begin{proposition}[\cite{Abdelkarim_Maroni_1997, Maroni_1991, Maroni_1999}] \label{prop2.3}
 Let $\{P_{n}(x)\}_{n\geq0}$ be a SMOP with respect to a linear functional $\U$ and let $\sigma(x)$ be a monic polynomial. $\U$ satisfies \eqref{pearson} if and only if there exists an integer $s\geq0$ such that
 \begin{equation*}
  \sigma(x)P^{[1,\nu^*]}_{n}(x) = \sum_{j=n-s}^{n+\deg(\sigma(x))} \lambda_{j,n}P_j(x), \,\, n\geq s, \quad\text{and}\quad \lambda_{n-s,n}\neq0, \,\, n\geq s+1.
 \end{equation*}
Equivalently,
 \begin{equation}\label{eq-prop2.3}
  \widetilde{\sigma}(x)P^{[1,\nu]}_{n}(x) = \sum_{j=n-s}^{n+\deg(\widetilde\sigma(x))} \widetilde{\lambda}_{j,n}P_j(x), \,\, n\geq s, \quad\text{and}\quad \widetilde{\lambda}_{n-s,n}\neq0, \,\, n\geq s+1.
 \end{equation}
In these equations, $\sigma(x)$ and $\widetilde{\sigma}$(x) are the polynomials appearing in Proposition \ref{U-semi-D-D*}, and $\lambda_{j,n}$ and $\widetilde{\lambda}_{j,n}$ are complex parameters for all $n$ and $j$.
\end{proposition}


\section{Main results}

\begin{definition}
 A pair of regular linear functionals $(\U,\V)$ is said to be a \emph{$(M,N)$-$D_{\nu}$-coherent pair of order $(m,k)$}, with fixed $M,N,m,k\in\mathbb{N}\cup\{0\}$, if their corresponding SMOP $\{P_n(x)\}_{n\geq0}$ and $\{Q_n(x)\}_{n\geq0}$ satisfy
 \begin{equation} \label{eq (M,N)-Dnu-coherence order (m,k) with [], Dnu=Dw or Dnu=Dq}
  P^{[m,\nu]}_{n}(x) + \sum_{i=1}^M a_{i,n} P^{[m,\nu]}_{n-i}(x) = Q^{[k,\nu]}_{n}(x) + \sum_{i=1}^N b_{i,n} Q^{[k,\nu]}_{n-i}(x), \quad n\geq 0,
 \end{equation}
 where $a_{i,n},b_{i,n}\in\mathbb{C}$, $a_{M,n}\neq0$ for $n\geq M$, $b_{N,n}\neq0$ for $n\geq N$, and $a_{i,n}=b_{i,n}=0$ if $i>n$. In addition, $(\U,\V)$ is said to be a \emph{$(M,N)$-$D_{\nu}$-coherent
  pair of order $m$} if it is a $(M,N)$-$D_{\nu}$-coherent pair of order $(m,0)$.
\end{definition}

In the next theorems, we state the $D_\nu$-analogue results obtained in \cite{Jesus_Marcellan_Petronilho_PinzonCortes201X, Jesus_Petronilho_2008, Marcellan_PinzonCortes_2012},
and we generalize the results stated in \cite{Area_Godoy_Marcellan_2000, Area_Godoy_Marcellan_2003, Kwon_Lee_Marcellan_2004, Marcellan_PinzonCortes_2012_Dw} for $\nu=\omega$, and in \cite{Area_Godoy_Marcellan_2002, Marcellan_PinzonCortes_2012_Dq} for $\nu=q$, respectively.
Moreover, we give a complete description of the $D_\nu$-semiclassical discrete orthogonal polynomials
in the framework of $(M,N)$-$D_\nu$-coherence of order $(m,k)$.

\begin{theorem} \label{main-teo}
 Let $(\U,\V)$ be a $(M,N)$-$D_{\nu}$-coherent pair of order $(m,k)$ given by \eqref{eq (M,N)-Dnu-coherence order (m,k) with [], Dnu=Dw or Dnu=Dq} with $m\geq k$. Let $\mathcal{L}_{M+N}=[l_{i,j}]_{i,j=0}^{M+N-1}$ be the following squared matrix of order $M+N$
 \begin{equation} \label{eq Matrix L_M+N, Dw y Dq}
  l_{i,j} = \left\{
   \begin{array}{ll}
    a_{j-i,j} & \text{if} \quad 0\leq i\leq N-1 \quad \text{and} \quad i\leq j\leq M+i, \\
    b_{j-i+N,j} & \text{if} \quad N\leq i\leq M+N-1 \quad \text{and} \quad i-N\leq j\leq i, \\
    0 & \text{otherwise,}
   \end{array}\right.
 \end{equation}
 with $a_{0,j_1}=b_{0,j_2}=1$, $0\leq j_1\leq N-1$, $0\leq j_2\leq M-1$. If $\det(\mathcal{L}_{M+N})\neq0$, then there exist polynomials $\phi_{M+k+n}(x;\nu)$ and $\psi_{N+m+n}(x;\nu)$, of
degrees $M+k+n$ and $N+m+n$, respectively,  such that
\begin{gather}
  D_{\nu^*}^{m-k}[\phi_{M+k+n}(x;\nu)\V]  = \psi_{N+m+n}(x;\nu)\U, \quad n\geq 0,
\label{eq Dnu^m-k[phiV]=psiU, Dw y Dq}
\end{gather}
and there exist polynomials $\varphi(x;\nu)$ and $\rho(x;\nu)$ such that
\begin{gather}
  \varphi(x;\nu)\U  =\rho(x;\nu)\V.
\label{rat-mod-uv}
\end{gather}

Furthermore
\begin{enumerate}
\item If $k=m$ then $\U$ is a $D_{\nu}$-semiclassical linear functional if and only if so is
$\V$.
\item If $m>k$, then $\U$ and $\V$ are both $D_{\nu}$-semiclassical linear functionals.
\end{enumerate}
\end{theorem}

\begin{proof}
 From \eqref{eq (M,N)-Dnu-coherence order (m,k) with [], Dnu=Dw or Dnu=Dq}, let $a_{0,n}=b_{0,n}=1$ and
 \begin{equation} \label{eq Rn(x;nu) = sum P^[m,nu] + sum Q^[k,nu], Dw y Dq}
  R_n(x;\nu) = \sum_{i=0}^M a_{i,n} P^{[m,\nu]}_{n-i}(x) = \sum_{i=0}^N b_{i,n} Q^{[k,\nu]}_{n-i}(x), \quad n\geq 0.
 \end{equation}
 Let us consider $\{\mathfrak{p}_{n}\}_{n\geq0}$, $\{\mathfrak{q}_{n}\}_{n\geq0}$, $\{\mathfrak{r}_{n,\nu}\}_{n\geq0}$, $\{\mathfrak{e}_{n,\nu}\}_{n\geq0}$ and $\{\mathfrak{h}_{n,\nu}\}_{n\geq0}$ be the dual bases of the SMOP $\{P_n(x)\}_{n\geq0}$, $\{Q_n(x)\}_{n\geq0}$ and the sequences $\{R_n(x;\nu)\}_{n\geq0}$, $\{P^{[m,\nu]}_n(x)\}_{n\geq0}$ and $\{Q^{[k,\nu]}_n(x)\}_{n\geq0}$, respectively. From
 \begin{align*}
  & \langle \mathfrak{e}_{n,\nu}, R_j(x;\nu) \rangle \overset{\eqref{eq Rn(x;nu) = sum P^[m,nu] + sum Q^[k,nu], Dw y Dq}}{=} \sum_{i=0}^M \langle \mathfrak{e}_{n,\nu}, a_{i,j} P^{[m,\nu]}_{j-i}(x) \rangle
  = \left\{
   \begin{array}{ll}
    a_{j-n,j} & \text{if} \quad n\leq j\leq n+M, \\
    0 & \text{ otherwise,}
   \end{array}
  \right. \\
  & \langle \mathfrak{h}_{n,\nu}, R_j(x;\nu) \rangle \overset{\eqref{eq Rn(x;nu) = sum P^[m,nu] + sum Q^[k,nu], Dw y Dq}}{=} \sum_{i=0}^N \langle \mathfrak{h}_{n,\nu}, b_{i,j} Q^{[k,\nu]}_{j-i}(x) \rangle
  = \left\{
   \begin{array}{ll}
    b_{j-n,j} & \text{if} \quad n\leq j\leq n+N, \\
    0 & \text{otherwise,}
   \end{array}
  \right.
 \end{align*}
 it follows that
 \begin{align}
  & \mathfrak{e}_{n,\nu} = \sum_{j\geq0} \langle \mathfrak{e}_{n,\nu}, R_j(x;\nu) \rangle \mathfrak{r}_{j,\nu} = \sum_{j=n}^{n+M} a_{j-n,j} \mathfrak{r}_{j,\nu}, \quad n\geq 0, \label{eq e_n,nu=sum a_ r_j,nu, Dw y Dq} \\
  & \mathfrak{h}_{n,\nu} = \sum_{j\geq0} \langle \mathfrak{h}_{n,\nu}, R_j(x;\nu) \rangle \mathfrak{r}_{j,\nu} = \sum_{j=n}^{n+N} b_{j-n,j} \mathfrak{r}_{j,\nu}, \quad n\geq 0. \label{eq h_n,nu=sum b_ r_j,nu, Dw y Dq}
 \end{align}
 Using \eqref{eq e_n,nu=sum a_ r_j,nu, Dw y Dq} and \eqref{eq h_n,nu=sum b_ r_j,nu, Dw y Dq} for $0\leq n\leq N-1$ and $0\leq n\leq M-1$, respectively, we set
 \begin{equation*}
  \mathcal{L}_{M+N} \left[
                      \begin{array}{c}
                        \mathfrak{r}_{0,\nu} \\ \vdots \\ \mathfrak{r}_{N-1,\nu} \\ \mathfrak{r}_{N,\nu} \\ \vdots \\ \mathfrak{r}_{N+M-1,\nu} \\
                      \end{array}
                    \right] =
                    \left[
                      \begin{array}{c}
                        \mathfrak{e}_{0,\nu} \\ \vdots \\ \mathfrak{e}_{N-1,\nu} \\ \mathfrak{h}_{0,\nu} \\ \vdots \\ \mathfrak{h}_{M-1,\nu} \\
                      \end{array}
                    \right],
 \end{equation*}
 where the matrix $\mathcal{L}_{M+N}$ is given by \eqref{eq Matrix L_M+N, Dw y Dq}. By assumption $\det(\mathcal{L}_{M+N})\neq0$, then we can solve this linear system and obtain, for $0\leq i\leq M+N-1$,
 \begin{equation} \label{eq r_i,nu = e_0,nu +...+ e_N-1,nu + h_0,nu +...+h_M-1,nu, Dw y Dq}
  \mathfrak{r}_{i,\nu} = \alpha_{i,0}\mathfrak{e}_{0,\nu} + \cdots + \alpha_{i,N-1}\mathfrak{e}_{N-1,\nu} + \alpha_{i,N}\mathfrak{h}_{0,\nu} + \cdots + \alpha_{i,N+M-1}\mathfrak{h}_{M-1,\nu},
 \end{equation}
 where $\alpha_{i,j}$, $0\leq j\leq N+M-1$, are some constants.
If,  for every $i\geq 0$, we multiply \eqref{eq e_n,nu=sum a_ r_j,nu, Dw y Dq} for $n=N+i$ by $b_{N,M+N+i}$, and \eqref{eq h_n,nu=sum b_ r_j,nu, Dw y Dq} for $n=M+i$ by $a_{M,M+N+i}$, and subtracting the resulting equations, we get
 \begin{multline} \label{eq e_N+i,nu - h_M+i,nu = r_min,nu +...+ r_M+N+i-1,nu, Dw y Dq}
  b_{N,M+N+i}\mathfrak{e}_{N+i,\nu} - a_{M,M+N+i}\mathfrak{h}_{M+i,\nu} \\
  = \beta_{1,i}\mathfrak{r}_{\min\{M,N\}+i,\nu} + \cdots + \beta_{\max\{M,N\},i}\mathfrak{r}_{M+N+i-1,\nu}, \quad i\geq 0,
 \end{multline}
 where $\beta_{j,i}$, $1\leq j\leq \max\{M,N\}$, $i\geq 0$, are constants. Additionally, for $t\geq 0$ fixed, using \eqref{eq e_n,nu=sum a_ r_j,nu, Dw y Dq} we can recursively obtain an expression for $\mathfrak{r}_{M+N+t,\nu}$ as a linear combination of $\mathfrak{r}_{i,\nu}$, $0\leq i\leq M+N-1$, and $\mathfrak{e}_{j,\nu}$, $N\leq j\leq N+t$, (since $a_{M,M+j}\neq0$, $N\leq j\leq N+t$). Hence, using \eqref{eq r_i,nu = e_0,nu +...+ e_N-1,nu + h_0,nu +...+h_M-1,nu, Dw y Dq}, \eqref{eq e_N+i,nu - h_M+i,nu = r_min,nu +...+ r_M+N+i-1,nu, Dw y Dq} becomes
 \begin{multline*}
  \tilde{\alpha}_{i,0}\mathfrak{e}_{0,\nu} + \cdots + \tilde{\alpha}_{i,N+i-1}\mathfrak{e}_{N+i-1,\nu} + b_{N,M+N+i}\mathfrak{e}_{N+i,\nu} \\
  = \tilde{\beta}_{i,0}\mathfrak{h}_{0,\nu} + \cdots + \tilde{\beta}_{i,M-1}\mathfrak{h}_{M-1,\nu} + a_{M,M+N+i}\mathfrak{h}_{M+i,\nu}, \quad i\geq 0,
 \end{multline*}
 where $\tilde{\alpha}_{i,j_1}$, $\tilde{\beta}_{i,j_2}$, for $0\leq j_1\leq N+i-1$, $0\leq j_2\leq M-1$, are constants. Applying the $m$th $D_{\nu}$-derivative $D^{m}_{\nu^*}$ and using \eqref{eq relacion base dual con U, y Dnu*^m e_n,nu = (-1)^m eta_n,nu p_n+m}, since $m\geq k$, we get
 \begin{multline*}
  \widehat{\alpha}_{i,0}\mathfrak{p}_{m} + \cdots + \widehat{\alpha}_{i,N+i-1}\mathfrak{p}_{N+i-1+m} + b_{N,M+N+i}(-1)^m \eta_{N+i,m,\nu}\mathfrak{p}_{N+i+m} = \\
  D_{\nu^*}^{m-k}\left[\widehat{\beta}_{i,0}\mathfrak{q}_{k} + \cdots + \widehat{\beta}_{i,M-1}\mathfrak{q}_{M-1+k} + a_{M,M+N+i} (-1)^k \eta_{M+i,k,\nu} \mathfrak{q}_{M+i+k}\right],
 \end{multline*}
 for $i\geq 0$. Therefore, from \eqref{eq relacion base dual con U, y Dnu*^m e_n,nu = (-1)^m eta_n,nu p_n+m} it follows \eqref{eq Dnu^m-k[phiV]=psiU, Dw y Dq} for all $n\geq 0$ with
 \begin{align*}
  & \phi_{M+k+n}(x;\nu) = (-1)^k \frac{\eta_{M+n,k,\nu} a_{M,M+N+n}}{\langle \V, Q_{M+k+n}^2(x) \rangle} x^{M+k+n} + \text{\small{\emph{lower degree terms}}}, 
  \\
  & \psi_{N+m+n}(x;\nu) = (-1)^m \frac{\eta_{N+n,m,\nu} b_{N,M+N+n}}{\langle \U, P_{N+m+n}^2(x) \rangle} x^{N+m+n} + \text{\small{\emph{lower degree terms}}}. 
 \end{align*}

Setting $k=m$ in equation \eqref{eq Dnu^m-k[phiV]=psiU, Dw y Dq} it follows that
$\U$ and ${\V}$ are connected by the rational modification \eqref{rat-mod-uv}
where $\rho(x;\nu)=\phi_{M+m+n}(x;\nu)$ and $\varphi(x;\nu)=\psi_{N+m+n}(x;\nu)$.
Therefore, by Proposition \ref{U-semi-lin},
$\U$ is a $D_{\nu}$-semiclassical linear functional if and only if so is $\V$.

Finally, let us consider $m>k$.  From \eqref{dw-pu} and \eqref{q-dw-pu},
\eqref{eq Dnu^m-k[phiV]=psiU, Dw y Dq} becomes, respectively, for each $n\geq 0$,
 \begin{gather*}
  \sum_{j=0}^{m-k} \binom{m-k}{j} \bigl(D^j_{-\omega} \phi_{M+k+n}\bigr)\bigl(x-(m-k-j)\omega;\omega\bigr) \, D^{m-k-j}_{-\omega}\V = \psi_{N+m+n}(x;\omega)\U, \\
  \sum_{j=0}^{m-k} \qbinom{m-k}{j}{q^{-1}} q^{-j}\bigl(D^j_{q^{-1}} \phi_{M+k+n}\bigr)\bigl(q^{-(m-k-j)}x;q\bigr) \, D^{m-k-j}_{q^{-1}}\V = \psi_{N+m+n}(x;q)\U.
 \end{gather*}
These equations, for $n=0,1,\ldots, m-k$, leads to the following systems (one for
$\nu=\omega$ and the other one for $\nu=q$) of functional linear equations
 \begin{equation*}
  \mathcal{T}_{m-k+1}(x;\nu) \left[
 \begin{array}{c}
  D_{\nu^*}^{m-k}\V \\ \vdots \\ D_{\nu^*}\V \\ \V \\
  \end{array}
 \right] = \left[
  \begin{array}{c}
 \psi_{N+m}(x;\nu)\U \\ \psi_{N+m+1}(x;\nu)\U \\ \vdots \\ \psi_{N+m+(m-k)}(x;\nu)\U \\
 \end{array}
 \right],
 \end{equation*}
 where $\det\left(\mathcal{T}_{m-k+1}(x;\nu)\right)\neq0$. Therefore, for $m>k$ we can solve
 the above systems with respect to $\V$ and $D_{\nu^*}\V$ (e.g., by using the Cramer's rule).
Solving it for $\V$ we obtain the relation \eqref{rat-mod-uv} where
$$
\rho(x;\nu) :=
 \det\left(\mathcal{T}_{m-k+1}(x;\nu)\right),
$$
 so that
 \begin{gather*}
  \rho(x;\omega) = \det\Biggl( \Bigl[ \bigl(D^i_{-\omega} \phi_{M+k+n}\bigr)\bigl(x-(m-k-i)\omega;\omega\bigr) \Bigr]_{i,n=0}^{m-k} \Biggr) \prod_{j=0}^{m-k} \binom{m-k}{j} \neq 0, \\
  \rho(x;q) = \det\Biggl( \Bigl[ \bigl(D^i_{q^{-1}} \phi_{M+k+n}\bigr)\bigl(q^{-(m-k-i)}x;q\bigr) \Bigr]_{i,n=0}^{m-k} \Biggr) \prod_{j=0}^{m-k} \qbinom{m-k}{j}{q^{-1}} q^{-j} \neq 0,
 \end{gather*}
and $\varphi(x;\nu)$ is a polynomial. In the same way, solving the system for
$D_{\nu^*}\V$ we obtain
$\rho(x;\nu) D_{\nu^*}\V = \varsigma(x;\nu)\U$, being $\varsigma(x;\nu)$ a polynomial. Thus,
 \begin{align*}
  D_{-\omega} & \left[\varphi(x+\omega;\omega)\rho(x+\omega;\omega)\V\right] 
  = \varphi(x;\omega)\varsigma(x;\omega)\U + D_{-\omega}\left[\varphi(x+\omega;\omega)\rho(x+\omega;\omega)\right]\V \\
  & = \left\{ \varsigma(x;\omega)\rho(x;\omega) + D_{-\omega}\left[\varphi(x+\omega;\omega)\rho(x+\omega;\omega)\right] \right\}\V, \\
  D_{q^{-1}} & \left[\varphi(qx;q)\rho(qx;q)\V\right]
  = \varphi(x;q)\varsigma(x;q)\U + q^{-1}D_{q^{-1}}\left[\varphi(qx;q)\rho(qx;q)\right]\V \\
  & = \left\{ \varsigma(x;q)\rho(x;q) + q^{-1}D_{q^{-1}}\left[\varphi(qx;q)\rho(qx;q)\right] \right\}\V, \\
 \end{align*}
 i.e.,  $\V$ is $D_{\nu^*}$-semiclassical linear functional. Then, using \eqref{rat-mod-uv}
 and Propositions \ref{U-semi-D-D*} and \ref{U-semi-lin} the result follows.
\end{proof}

\subsection{The special case $m=k+1$}

Let us consider now the special case when $m=k+1$. In this case Theorem \ref{main-teo}
gives that both $\U$ and $\V$ are $D_\nu$-semiclassical functionals and are connected by
the linear relation \eqref{rat-mod-uv}. Let us now discuss the inverse statement.

\begin{theorem}
 Let $\U$ and $\V$ be two $D_\nu$-semiclassical linear functionals related by a rational factor, i.e., there exist
 monic polynomials $\sigma(x)$ and $\varphi(x)$, and nonzero polynomials $\tau(x)$ and $\rho(x)$, such that
 \begin{gather*}
  D_{\nu^*}\left[\sigma(x)\V\right]=\tau(x)\V, \quad\text{and}\quad \varphi(x)\U = \rho(x)\V, \\
  \deg(\sigma(x))=\ell, \quad \deg(\tau(x))=t\geq1, \quad \deg(\varphi(x))=\jmath\,, \quad \deg(\rho(x))=r,
 \end{gather*}
 hold, and let $\{P_n(x)\}_{n\geq0}$ and $\{Q_n(x)\}_{n\geq0}$ be the SMOP associated with $\U$ and $\V$, respectively. Then,
 \begin{equation} \label{eq converse rational-semi impli coherence, Dw y Dq}
  \sum_{i=n-r-\ell}^{n+\jmath+\ell} a_{i,n} P_i^{[1,\nu]}(x) = \sum_{i=n-\jmath-s}^{n+\jmath+\ell} b_{i,n} Q_i(x),
 \end{equation}
 where $a_{n+\jmath+\ell,n}b_{n+\jmath+\ell,n}\neq0$, for $n\geq0$, and $s=\max\{\ell-2,t-1\}$. Therefore, $(\U,\V)$ is a $(\jmath+2\ell+r,2\jmath+\ell+s)$-$D_\nu$-coherent pair of order $1$.
\end{theorem}
\begin{proof}
Let us prove the $q$-case. The proof for the case of $D_\omega$ is similar.

 From Eq. \eqref{eq-prop2.3} of Proposition \ref{prop2.3}, it follows that
 \begin{equation} \label{eq sigmaQ^[1,nu] = sum xi_(i,n,1) Q_i, Dw y Dq}
  \sigma(x)Q^{[1,q]}_{n}(x) = \sum_{i=n-s}^{n+\ell} \xi_{i,n,1}Q_i(x), \,\,\, n\geq s, \quad \xi_{n-s,n,1}\neq0, \,\, n\geq s+1.
 \end{equation}
Using $\varphi(x)\U = \rho(x)\V$, for $n\geq0$, $\varphi(x) Q_n(x) = \sum_{i=0}^{n+\jmath} \xi_{i,n,2} P_i(x)$, where
$$
\langle\U,P^2_i(x)\rangle \xi_{i,n,2} = \langle \U, \varphi(x) Q_n(x)P_i(x)\rangle = \langle \rho(x)\V, Q_n(x)P_i(x)\rangle = 0
$$
 for $i+r\leq n-1$, one finds
 \begin{equation} \label{eq varphiQ = sum xi_(i,n,2) P_i, Dw y Dq}
  \varphi(x) Q_n(x) = \sum_{i=n-r}^{n+\jmath} \xi_{i,n,2} P_i(x), \quad n\geq r.
 \end{equation}
 Furthermore,
 \begin{gather}
  \sigma(q^{-1}x)P_n(x) = \sum_{i=n-\ell}^{n+\ell} \xi_{i,n,3} P_i(x), \quad n\geq\ell, \label{eq sigma(xnu*)P = sum xi_(i,n,3) P, Dw y Dq} \\
  D_q\left[\varphi(x)\sigma(q^{-1}x)\right] Q_{n+1}(x) = \sum_{i=n-\jmath-\ell+2}^{n+\jmath+\ell} \xi_{i,n,4} Q_i(x), \quad n\geq\jmath+\ell-2,
  \label{eq Dnu[varphisigma(xnu*)] Q = sum xi_(i,n,4) Q, Dw y Dq} \\
  \varphi(qx)Q_n(x) = \sum_{i=n-\jmath}^{n+\jmath} \xi_{i,n,5} Q_i(x), \quad n\geq\jmath\,, \label{eq varphi(xnu)Q = sum xi_(i,n,5) Q, Dw y Dq}
 \end{gather}
 where $\langle\U,P^2_i(x)\rangle \xi_{i,n,3}=\bigl\langle\U,\sigma(q^{-1}x)P_n(x)P_i(x)\bigr\rangle$,
 $\langle\V,Q^2_i(x)\rangle\xi_{i,n,4}=\bigl\langle\V,$\linebreak$D_q\left[\varphi(x)\sigma(q^{-1}x)\right] Q_{n+1}(x)Q_i(x)\bigr\rangle$ and
 $\langle\V,Q^2_i(x)\rangle\xi_{i,n,5}=\bigl\langle\V,\varphi(qx)Q_n(x)Q_i(x)\bigr\rangle$. On the other hand,
 \begin{multline} \label{eq Dnu[varphisigma(xnu*)Q] = ...., Dw y Dq}
  D_q\left[\varphi(x)\sigma(q^{-1}x)Q_{n+1}(x)\right] 
  = D_q\left[\varphi(x)\sigma(q^{-1}x)\right]Q_{n+1}(x) + \varphi(qx)\sigma(x) D_q\left[Q_{n+1}(x)\right].
 \end{multline}
 Let us compute each term in the previous $q$-derivative
 \begin{multline*}
  D_q\left[\varphi(x)\sigma(q^{-1}x)Q_{n+1}(x)\right]  \underset{\eqref{eq sigma(xnu*)P = sum xi_(i,n,3) P, Dw y Dq}}{\overset{\eqref{eq varphiQ = sum xi_(i,n,2) P_i, Dw y Dq}}{=}} \sum_{i=n+1-r}^{n+1+\jmath} \xi_{i,n+1,2} \sum_{j=i-\ell}^{i+\ell} \xi_{j,i,3} D_q\left[P_j(x)\right] \\
  = \sum_{i=n-r-\ell+1}^{n+\jmath+\ell+1} \xi_{i,n,6} \frac{D_q\left[P_i(x)\right]}{\eta_{i-1,1,q}} = \sum_{i=n-r-\ell}^{n+\jmath+\ell} \xi_{i+1,n,6} P_i^{[1,q]}(x),
 \end{multline*}
 \begin{align*}
  \varphi(qx)\sigma(x) D_q\left[Q_{n+1}(x)\right]  & \underset{\eqref{eq varphi(xnu)Q = sum xi_(i,n,5) Q, Dw y Dq}}{\overset{\eqref{eq sigmaQ^[1,nu] = sum xi_(i,n,1) Q_i, Dw y Dq}}{=}} \eta_{n,1,q} \sum_{i=n-s}^{n+\ell} \xi_{i,n,1} \sum_{j=i-\jmath}^{i+\jmath} \xi_{j,i,5} Q_j(x) 
  = \sum_{i=n-s-\jmath}^{n+\ell+\jmath} \xi_{i,n,7} Q_i(x).
 \end{align*}
 Consequently, from \eqref{eq Dnu[varphisigma(xnu*)] Q = sum xi_(i,n,4) Q, Dw y Dq} and taking into account that $s\geq\ell-2$, \eqref{eq Dnu[varphisigma(xnu*)Q] = ...., Dw y Dq} becomes \eqref{eq converse rational-semi impli coherence, Dw y Dq}.
\end{proof}

Before concluding this section we would like to remark that an interesting question concerning the study presented in
this work is finding non-trivial examples illustrating the developed theory. This appears to be an hard task from a technical
point of view, and some examples are now under construction (which we hope to be the subject of further work)
following ideas presented in previous works on coherent pairs of OPs, not only motivated by the continuous case, but also by
the $q-$case. For instance, an important source of motivation is Section 6 contained in the paper
\cite{Area_Godoy_Marcellan_2002} by I. Area, E. Godoy, and F. Marcell\'an, where these authors present very
interesting examples, giving the classification of all $q-$coherent pairs of positive-definite linear
functionals when one of them is either the little $q-$Jacobi linear functional or the little $q-$Laguerre linear functional.
With this respect see also the more recent work \cite{Marcellan_PinzonCortes_2012_Dq}.

\section{Application to $D_{\nu}$-Sobolev Orthogonal Polynomials}

In the following $\mathbb{P}$ will denote the linear space of polynomials with real coefficients
and $\U$ and $\V$ will be two positive definite linear functionals.
We will consider the Sobolev-type inner product, for fixed $m\geq1$,
\begin{equation} \label{eq Sobolev inner prod_lambda,nu, Dw y Dq}
 \left\langle p(x),r(x) \right\rangle_{\lambda,\nu} = \left\langle \U, p(x)r(x) \right\rangle  + \lambda \left\langle \V, (D^m_\nu p)(x)(D^m_\nu r)(x) \right\rangle, \quad \lambda>0,
\end{equation}
where $\U$ and $\V$ are regular linear functionals
(which includes the special cases of discrete measures supported on either a uniform
lattice, when $\nu=\omega$, or a $q$-lattice, when $\nu=q$).
Let $\{P_n(x)\}_{n\geq0}$, $\{Q_n(x)\}_{n\geq0}$ and $\{S_n(x;\lambda,\nu)\}_{n\geq0}$ be the
SMOP with respect to $\U$, $\V$ and $\langle\cdot\,,\cdot\rangle_{\lambda,\nu}$, respectively.

\begin{remark}
Notice that we have assumed that both $\U$ and $\V$ are positive definite regular linear functionals.
Otherwise, we could not guarantee a priori that the bilinear form
$\left\langle \cdot,\cdot \right\rangle_{\lambda,\nu}$  defined by
\eqref{eq Sobolev inner prod_lambda,nu, Dw y Dq} is an inner-product.
This is an interesting open problem for a further
investigation but it is beyond the study presented here. Since we are interested in
showing that the notion of coherence is crucial in finding the polynomials
$\{S_n(x;\lambda,\nu)\}_{n\geq0}$ the assumption that both  $\U$ and $\V$ are positive definite
functionals is a sufficient condition for the  Sobolev-type inner product \eqref{eq Sobolev inner prod_lambda,nu, Dw y Dq} to be well defined.
\end{remark}

\begin{proposition}
 The following algebraic relations hold
 \begin{equation} \label{eq Q_n=sum_0^n P^[m,nu]_j, Dw y Dq}
   Q_{n}(x) = P^{[m,\nu]}_{n}(x) + \sum_{j=0}^{n-1} \frac{\eta_{j,m,\nu}}{\eta_{n,m,\nu}} \frac{\langle \U, T_{n+m}(x;\nu)P_{j+m}(x)\rangle}{\langle \U, P^2_{j+m}(x)\rangle} P^{[m,\nu]}_{j}(x), \quad n\geq 0,
 \end{equation}
 \begin{multline} \label{eq S_n + sum_m^n-1 S_i = P_n + sum_m^n-1 P_i, Dw y  Dq}
   S_{n}(x;\lambda,\nu) + \sum_{i=m}^{n-1} \frac{\langle \U, T_{n}(x;\nu)S_{i}(x;\lambda,\nu)\rangle S_{i}(x;\lambda,\nu)}{\langle S_{i}(x;\lambda,\nu), S_{i}(x;\lambda,\nu) \rangle_{\lambda,\nu}} \\
   = P_n(x) + \sum_{i=m}^{n-1} \frac{\langle \U, T_{n}(x;\nu)P_{i}(x)\rangle P_{i}(x)}{\langle\U, P_{i}^2(x) \rangle}, \quad n\geq m,
 \end{multline}
 and $S_n(x;\lambda,\nu)=P_n(x)$ for $n\leq m$, where
 \begin{equation} \label{eq lim S_n = T_n, Dw y Dq}
  T_n(x;\nu) = \lim_{\lambda\longrightarrow\infty} S_n(x;\lambda,\nu) \, , \quad n\geq 0.
 \end{equation}
\end{proposition}
\begin{proof}
 From \eqref{eq Sobolev inner prod_lambda,nu, Dw y Dq}, $\langle P_n(x), x^i\rangle_{\lambda,\nu} = 0$, for $i<n<m$, and thus $S_n(x;\lambda,\nu)=P_n(x)$ for $n<m$. Also, from the uniqueness of the SMOP with respect to the bilinear functional $\W$ associated with $\langle\cdot,\cdot\rangle_{\lambda,\nu}$, each $S_n(x;\lambda,\nu)$ can be written as
 \begin{equation*}
  S_n(x;\lambda,\nu) = \frac{
  \left|
   \begin{array}{cccc}
     w_{0,0,\nu} & \cdots & w_{0,n-1,\nu} & w_{0,n,\nu} \\
     \vdots & \ddots & \vdots & \vdots \\
     w_{n-1,0,\nu} & \cdots & w_{n-1,n-1,\nu} & w_{n-1,n,\nu} \\
     1 & \cdots & x^{n-1} & x^n \\
   \end{array}
  \right|}{\det\left([w_{i,j,\nu}]_{i,j=0}^{n-1}\right)}, \,n\geq 1, \,\, S_0(x;\lambda,\nu)=1,
 \end{equation*}
 where $w_{i,j,\nu}=\langle x^i,x^j\rangle_{\lambda,\nu} = u_{i+j} + \lambda \eta_{i-m,m,\nu}\eta_{j-m,m,\nu} v_{(i-m)+(j-m)}$, for $i,j\geq 0$. Hence, every coefficient of $S_n(x;\lambda,\nu)$ is a rational
 function of $\lambda$ such that their numerator and denominator have the same degree, and as a
 consequence, there exist the monic polynomials $T_n$ given by \eqref{eq lim S_n = T_n, Dw y Dq}.
 On the other hand, from \eqref{eq lim S_n = T_n, Dw y Dq} and
 \eqref{eq Sobolev inner prod_lambda,nu, Dw y Dq} we obtain, for $n\geq 0$,
 \begin{equation} \label{help45}
  \langle \U, T_n(x;\nu)x^i\rangle = 0, \,\, i<\min\{n,m\}, \quad \langle \V, D_\nu^m[T_n(x;\nu)]\,x^j\rangle = 0, \,\, j<n-m.
 \end{equation}
 Indeed, for $i<\min\{n,m\}$, $\langle S_n(x;\lambda,\nu), x^i\rangle_{\lambda,\nu} =0$ and
 $D_\nu^m(x^i)$, hence
 \begin{gather*}
  \langle \U, T_n(x;\nu)x^i\rangle = \lim_{\lambda\rightarrow\infty} \left[\langle S_n(x;\lambda,\nu), x^i\rangle_{\lambda,\nu} - \lambda \langle \V, D_\nu^m(S_n(x;\lambda,\nu)) D_\nu^m(x^i) \rangle\right] = 0.
\end{gather*}
For $j<n-m$, we write $x^j= D_\nu \pi_{j+m}(x;\nu)$ for a certain polynomial
$\pi_{j+m}$ of degree $m+j$. Therefore, using \eqref{eq lim S_n = T_n, Dw y Dq}
and taking $p(x)=T_n(x;\nu)$ and $r(x)=\pi_{j+m}(x;\nu)$ in
\eqref{eq Sobolev inner prod_lambda,nu, Dw y Dq} we get
$$
\left\langle \V, (D^m_\nu T_n)(x;\nu) x^j \right\rangle=\lim_{\lambda\to\infty}
\left\langle \V, (D^m_\nu S_n)(x;\lambda,\nu)(D^m_\nu \pi_{j+m})(x;\nu) \right\rangle
$$
$$
=\lim_{\lambda\to\infty}\frac{1}{\lambda}
\left[  \left\langle S_n(x;\lambda,\nu),  \pi_{j+m}(x;\nu) \right\rangle_{\lambda,\nu} -
\left\langle \U, S_n(x;\lambda,\nu) \pi_{j+m}(x;\nu) \right\rangle  \right].
$$
The first term is zero when $j<n-m$ and for the second one we have
$$
\lim_{\lambda\to\infty}\frac{1}{\lambda} \left\langle \U, S_n(x;\lambda,\nu) \pi_{j+m}(x;\nu) \right\rangle
=0,
$$
since, by \eqref{eq lim S_n = T_n, Dw y Dq}, the limit $\lim_{\lambda\to\infty} \left\langle \U, S_n(x;\lambda,\nu) \pi_{j+m}(x;\nu) \right\rangle$ exists. This proves \eqref{help45}.

From \eqref{help45}, it follows that, for $n\geq m$ and $n\geq 0$, respectively,
 \begin{gather}
   T_{n}(x;\nu) = \sum_{i=0}^{n} \frac{\langle \U, T_{n}(x;\nu)P_{i}(x)\rangle}{\langle \U, P_{i}^2(x) \rangle} P_{i}(x) = \sum_{j=0}^{n-m} \frac{\langle \U, T_{n}(x;\nu)P_{j+m}(x)\rangle}{\langle\U,P_{j+m}^2(x)\rangle} P_{j+m}(x), \notag \\
   \frac{D_\nu^m[T_{n+m}(x;\nu)]}{\eta_{n,m,\nu}} = \sum_{i=0}^{n} \frac{\langle \V, Q_{i}(x) D_\nu^m[T_{n+m}(x;\nu)]/\eta_{n,m,\nu} \rangle}{\langle\V,Q_{i}^2(x)\rangle} Q_{i}(x) = Q_{n}(x), \label{eq Q_n = Dnu^mT_n+m / eta_n,m,nu, Dw y Dq}
 \end{gather}
 which proves \eqref{eq Q_n=sum_0^n P^[m,nu]_j, Dw y Dq}. Finally, for the proof of \eqref{eq S_n + sum_m^n-1 S_i = P_n + sum_m^n-1 P_i, Dw y  Dq}, using \eqref{eq Sobolev inner prod_lambda,nu, Dw y Dq} and \eqref{help45} we get
 \begin{align*}
  T_{n}(x;\nu) &= \sum_{i=0}^{n} \frac{\langle T_{n}(x;\nu),S_{i}(x;\lambda,\nu)\rangle_{\lambda,\nu}}{\langle S_{i}(x;\lambda,\nu), S_{i}(x;\lambda,\nu) \rangle_{\lambda,\nu}} S_{i}(x;\lambda,\nu) \\
  & = S_{n}(x;\lambda,\nu) + \sum_{i=m}^{n-1} \frac{\langle\U, T_{n}(x;\nu)S_{i}(x;\lambda,\nu)\rangle}{\langle S_{i}(x;\lambda,\nu), S_{i}(x;\lambda,\nu) \rangle_{\lambda,\nu}} S_{i}(x;\lambda,\nu), \quad n\geq 0.
 \end{align*}
\end{proof}

Now, we will study the case when $\U$ and $\V$ form a $(M,N)$-$D_{\nu}$-coherent pair of order $m$, i.e,
when their corresponding SMOP $\{P_n(x)\}_{n\geq0}$ and $\{Q_n(x)\}_{n\geq0}$ satisfy
\begin{equation} \label{eq (M,N)-Dnu-coherence order (m,0) with [], Dw y Dq}
 P^{[m,\nu]}_{n}(x) + \sum_{i=1}^M a_{i,n} P^{[m,\nu]}_{n-i}(x) = Q_{n}(x) + \sum_{i=1}^N b_{i,n} Q_{n-i}(x), \quad n\geq 0,
\end{equation}
where $a_{M,n}\neq0$ if $n\geq M$, $b_{N,n}\neq0$ if $n\geq N$, and $a_{i,n}=b_{i,n}=0$ when $i>n$.

One of the most important problems in the theory of Sobolev OPS is to find the explicit expressions
for the polynomials themselves. When the Sobolev OPS is orthogonal with respect to the inner product
\eqref{eq Sobolev inner prod_lambda,nu, Dw y Dq} and $\U$ and $\V$ constitute a $(M,N)$-$D_\nu$-coherent
pair of order $m$, it is possible to obtain the Sobolev orthogonal polynomials by using the following two
theorems that generalize an algebraic property proved for $(M,0)$-$D_\omega$-coherent and $(1,1)$-$D_\nu$-coherent pairs of order $1$, in \cite{Kwon_Lee_Marcellan_2004, Marcellan_PinzonCortes_2012_Dw, Marcellan_PinzonCortes_2012_Dq}, to $(M,N)$-$D_\nu$-coherent pairs of order $m$, and they are the $D_\nu$-analogue results obtained in \cite{Jesus_Marcellan_Petronilho_PinzonCortes201X, Jesus_Petronilho_2013}.

\begin{theorem} \label{teo (M,N)-Dnu-coherence order (m,0) impl relacion con Sobolev, Dw y Dq}
 Let $(\U,\V)$ be a $(M,N)$-$D_\nu$-coherent pair of order $m$ given by \eqref{eq (M,N)-Dnu-coherence order (m,0) with [], Dw y Dq}, and $K=\max\{M,N\}$. Then, $S_{n}(x;\lambda,\nu)=P_n(x)$ for $n< m$ and
 \begin{equation} \label{eq P_n+m + sum^M a P_n-i+m = S_n+m + sum^K c S_n-i+m, Dw y Dq}
  P_{n+m}(x) + \sum_{i=1}^M \frac{\eta_{n,m,\nu} a_{i,n}}{\eta_{n-i,m,\nu}} P_{n-i+m}(x) = S_{n+m}(x;\lambda,\nu) + \sum_{j=1}^{K} c_{j,n,\lambda,\nu} S_{n-j+m}(x;\lambda,\nu),
 \end{equation}
 for $n\geq 0$, where $c_{j,n,\lambda,\nu}= 0$ for $n<j\leq K$, and, for $1\leq j\leq K$,
 \begin{multline} \label{eq c_j,n,lambda,nu <S_n-j+m,S_n-j+m>_lambda,nu = sum^M + sum^N, Dw y Dq}
  c_{j,n,\lambda,\nu} = \frac{\eta_{n,m,\nu}}{\langle S_{n-j+m}(x;\lambda,\nu), S_{n-j+m}(x;\lambda,\nu) \rangle_{\lambda,\nu}} \Biggl[\sum_{i=j}^M \frac{a_{i,n}}{\eta_{n-i,m,\nu}} \\
  \bigl\langle \U, P_{n-i+m}(x)S_{n-j+m}(x;\lambda,\nu) \bigr\rangle + \lambda \sum_{i=j}^N b_{i,n} \bigl\langle \V, Q_{n-i}(x)D_\nu^m[S_{n-j+m}(x;\lambda,\nu)] \bigr\rangle \Biggr].
 \end{multline}
 Besides, for each $n\geq K$,
 \begin{enumerate}
  \item[\textbf{(i)}] if $M>N$ and $a_{M,n}\neq0$, then $c_{K,n,\lambda,\nu}\neq0$,
  \item[\textbf{(ii)}] if $M<N$ and $b_{N,n}\neq0$, then $c_{K,n,\lambda,\nu}\neq0$,
  \item[\textbf{(iii)}] if $M=N(=K)$ and $a_{M,n}b_{N,n}\neq0$ then,
     \begin{equation*}
      c_{K,n,\lambda,\nu}\neq0 \,\text{ iff }\, a_{K,n}\langle\U,P_{n-K+m}^2(x)\rangle + \lambda \eta_{n-K,m,\nu}^2 b_{K,n} \langle\V,Q_{n-K}^2(x)\rangle \neq 0.
     \end{equation*}
 \end{enumerate}

 Conversely, if there exist constants $\{c_{j,n,\lambda,\nu}\}_{n\geq0}$, $1\leq j\leq K$, and $\{a_{i,n}\}_{n\geq0}$, $1\leq i\leq M$, with $c_{j,n,\lambda,\nu}= 0$, $n-j+m<0$, and $a_{i,n}= 0$, $n-i+m<0$, such that \eqref{eq P_n+m + sum^M a P_n-i+m = S_n+m + sum^K c S_n-i+m, Dw y Dq} holds, then $(\U,\V)$ is a $(M,K)$-$D_\nu$-coherent pair of order $m$ given by
 \begin{equation} \label{eq (M,K)-Dnu-coherence (m,0) when sum Pn = sum Sn, Dw y dq}
  P^{[m,\nu]}_{n}(x) + \sum_{i=1}^M a_{i,n} P^{[m,\nu]}_{n-i}(x) = Q_n(x) + \sum_{j=1}^K b_{j,n} Q_{n-j}(x), \quad n\geq0,
 \end{equation}
 (whenever $b_{K,n}\neq0$ for $n\geq K$), where $b_{j,n}=0$ for $n<j\leq K$, and for $n\geq0$,
 \begin{equation} \label{eq b_j,n when sum Pn = sum Sn imply (M,K)-Dnu-coherence (m,0), Dw y Dq}
  b_{j,n} = \frac{\left\langle \V, \left(P^{[m,\nu]}_{n}(x) + \sum_{i=1}^M a_{i,n} P^{[m,\nu]}_{n-i}(x)\right)Q_{n-j}(x) \right\rangle}{\langle\V,Q_{n-j}^2(x)\rangle},\quad 1\leq j\leq\min\{K,n\}.
 \end{equation}
\end{theorem}
\begin{proof}
 $S_n(x;\lambda,\nu)=P_n(x)$, $n<m$, follows from $\langle P_n(x), x^i\rangle_{\lambda,\nu} = 0$, $i<n<m$. On the other hand, substituting \eqref{eq Q_n = Dnu^mT_n+m / eta_n,m,nu, Dw y Dq} in \eqref{eq (M,N)-Dnu-coherence order (m,0) with [], Dw y Dq}, and then, computing $D_\nu$-antiderivatives $m$ times (this is, a function $F(x)$ is a \emph{$D_\nu$-antiderivative} of a function $f(x)$ if $D_\nu F(x) = f(x)$, \cite{Kac_Cheung_2002}), we obtain for $n\geq 0$,
 \begin{equation*}
  \frac{P_{n+m}(x)}{\eta_{n,m,\nu}} + \sum_{i=1}^M \frac{a_{i,n} P_{n-i+m}(x)}{\eta_{n-i,m,\nu}} = \frac{T_{n+m}(x;\nu)}{\eta_{n,m,\nu}} + \sum_{i=1}^N \frac{b_{i,n} T_{n-i+m}(x;\nu)}{\eta_{n-i,m,\nu}} + \sum_{j=0}^{m-1} \kappa_{n,j}x^j.
 \end{equation*}
 Taking $\langle x^i\U,\cdot\,\rangle$, $i<m$, from \eqref{help45}, we get the linear system $\sum_{j=0}^{m-1} \kappa_{n,j}u_{j+i}=0$, $i=0,\ldots,m-1$. Since $\det\left([u_{i+j}]_{i,j=0}^{m-1}\right)\neq 0$, then $\kappa_{n,j}=0$, $j=0,\ldots,m-1$, $n\geq 0$. Hence, for $n\geq 0$,
 \begin{equation} \label{eq P_n+m + sum^M a P_n-i+m = T_n+m + sum^N b T_n-i+m, Dw y Dq}
  \frac{P_{n+m}(x)}{\eta_{n,m,\nu}} + \sum_{i=1}^M a_{i,n} \frac{P_{n-i+m}(x)}{\eta_{n-i,m,\nu}} = \frac{T_{n+m}(x;\nu)}{\eta_{n,m,\nu}} + \sum_{i=1}^N b_{i,n} \frac{T_{n-i+m}(x;\nu)}{\eta_{n-i,m,\nu}}.
 \end{equation}
 Furthermore, for $n\geq 0$,
 \begin{equation*}
  \frac{T_{n+m}(x;\nu)}{\eta_{n,m,\nu}} + \sum_{i=1}^N b_{i,n} \frac{T_{n-i+m}(x;\nu)}{\eta_{n-i,m,\nu}} = \frac{S_{n+m}(x;\lambda,\nu)}{\eta_{n,m,\nu}} + \sum_{j=1}^{n+m} \frac{c_{j,n,\lambda,\nu}}{\eta_{n,m,\nu}} S_{n-j+m}(x;\lambda,\nu),
 \end{equation*}
 where from \eqref{eq Sobolev inner prod_lambda,nu, Dw y Dq}, \eqref{eq P_n+m + sum^M a P_n-i+m = T_n+m + sum^N b T_n-i+m, Dw y Dq} and \eqref{eq Q_n = Dnu^mT_n+m / eta_n,m,nu, Dw y Dq}, for $1\leq j\leq n+m$,
 \begin{multline*}
  \langle S_{n-j+m}(x;\lambda,\nu),S_{n-j+m}(x;\lambda,\nu)\rangle_{\lambda,\nu} \frac{c_{j,n,\lambda,\nu}}{\eta_{n,m,\nu}} = \sum_{i=1}^M \frac{a_{i,n}}{\eta_{n-i,m,\nu}} \\
  \bigl\langle \U, P_{n-i+m}(x)S_{n-j+m}(x;\lambda,\nu) \bigr\rangle + \lambda \sum_{i=1}^N b_{i,n} \bigl\langle \V, Q_{n-i}(x) D_\nu^m[S_{n-j+m}(x;\lambda,\nu)] \bigr\rangle,
 \end{multline*}
 then $c_{j,n,\lambda,\nu}=0$ for $j>i$ or $j>K$. Thus, \eqref{eq P_n+m + sum^M a P_n-i+m = S_n+m + sum^K c S_n-i+m, Dw y Dq} and \eqref{eq c_j,n,lambda,nu <S_n-j+m,S_n-j+m>_lambda,nu = sum^M + sum^N, Dw y Dq} hold. Also, for $n\geq K$,
 \begin{equation*}
  \frac{c_{K,n,\lambda,\nu}}{\eta_{n,m,\nu}} = \frac{ \frac{a_{M,n}}{\eta_{n-M,m,\nu}} \langle\U,P_{n-M+m}^2(x)\rangle \delta_{M,K} + \lambda \eta_{n-N,m,\nu} b_{N,n} \langle\V,Q_{n-N}^2(x)\rangle \delta_{N,K} }{ \langle S_{n-K+m}(x;\lambda,\nu), S_{n-K+m}(x;\lambda,\nu)\rangle_{\lambda,\nu} },
 \end{equation*}
 holds, and as a consequence, \emph{(i)}, \emph{(ii)} and \emph{(iii)} follow. Finally,
 \begin{equation*}
  P^{[m,\nu]}_{n}(x) + \sum_{i=1}^M a_{i,n} P^{[m,\nu]}_{n-i}(x) = Q_n(x) + \sum_{j=1}^n b_{j,n} Q_{n-j}(x), \quad n\geq0,
 \end{equation*}
 with $b_{j,n}$, for $1\leq j\leq n$, given by \eqref{eq b_j,n when sum Pn = sum Sn imply (M,K)-Dnu-coherence (m,0), Dw y Dq}. Applying $\langle \,\cdot\,, p(x) \rangle_{\lambda,\nu}$ to both sides of \eqref{eq P_n+m + sum^M a P_n-i+m = S_n+m + sum^K c S_n-i+m, Dw y Dq}, for $p\in\mathbb{P}_{n-K+m-1}$, it follows that
 \begin{equation*}
  0 = \lambda \left\langle \V, \left(D_\nu^m[P_{n+m}(x)] + \sum_{i=1}^M \frac{\eta_{n,m,\nu} a_{i,n}}{\eta_{n-i,m,\nu}} D_\nu^m[P_{n-i+m}(x)]\right) D_\nu^m[p(x)] \right\rangle,
 \end{equation*}
 i.e., 
 \begin{equation*}
  0 = \left\langle \V, \left(P^{[m,\nu]}_{n}(x) + \sum_{i=1}^M a_{i,n} P^{[m,\nu]}_{n-i}(x)\right) r(x) \right\rangle, \quad \forall \, r\in\mathbb{P}_{n-K-1},
 \end{equation*}
 thus $b_{j,n}=0$, for $n-j\leq n-(K+1)$, which proves \eqref{eq (M,K)-Dnu-coherence (m,0) when sum Pn = sum Sn, Dw y dq}.
\end{proof}

\begin{theorem} \label{teo (M,N)-Dnu-coherence order (m,0) impl eq recursiva s_n,nu y c_j,n,lambda,nu, Dw y Dq}
 Let $(\U,\V)$ be a $(M,N)$-$D_\nu$-coherent pair of order $m$ given by \eqref{eq (M,N)-Dnu-coherence order (m,0) with [], Dw y Dq}, $K=\max\{M,N\}$, and for $n\geq 0$,
 \begin{equation*}
  s_{n,\nu}=\langle S_n(x;\lambda,\nu), S_n(x;\lambda,\nu)\rangle_{\lambda,\nu}, \quad \widetilde{a}_{i,n} = \frac{\eta_{n,m,\nu}}{\eta_{n-i,m,\nu}}a_{i,n}, \quad \widetilde{b}_{i,n} = \eta_{n,m,\nu}b_{i,n},
 \end{equation*}
 with $a_{i,n}=b_{i,n}=0$ if $i>n$, and, $a_{0,n}=b_{0,n}=1$ for $n\geq 0$. Then
 \begin{equation} \label{eq s*c = zeta_j,n,lambda,nu - sum^K-j c*c*s, Dw y Dq}
  s_{n+m,\nu} c_{j,n+j,\lambda,\nu} = \zeta_{j,n,\lambda,\nu} - \sum_{\ell=1}^{K-j} c_{\ell,n,\lambda,\nu} c_{j+\ell,n+j,\lambda,\nu} s_{n-\ell+m,\nu}, \quad 0\leq j\leq K, \, n\geq 0,
 \end{equation}
 with $s_{n,\nu}=\langle\U,P_n^2(x)\rangle$ for $n<m$, $c_{0,n,\lambda,\nu} = 1$ for $n\geq 0$, $c_{j,n,\lambda,\nu}= 0$ for $n<j\leq K$, and for $0\leq j\leq K$,
 \begin{equation*} 
  \zeta_{j,n,\lambda,\nu} = \sum_{i=j}^M \widetilde{a}_{i,n+j} \widetilde{a}_{i-j,n} \langle\U,P_{n+j-i+m}^2(x)\rangle + \lambda \sum_{i=j}^N \widetilde{b}_{i,n+j} \widetilde{b}_{i-j,n} \langle\V,Q_{n+j-i}^2(x)\rangle.
 \end{equation*}
\end{theorem}
\begin{proof}
 Notice that \eqref{eq P_n+m + sum^M a P_n-i+m = S_n+m + sum^K c S_n-i+m, Dw y Dq} and \eqref{eq c_j,n,lambda,nu <S_n-j+m,S_n-j+m>_lambda,nu = sum^M + sum^N, Dw y Dq} hold setting $c_{0,n,\lambda,\nu} = 1$ for $n\geq 0$. Then, from \eqref{eq (M,N)-Dnu-coherence order (m,0) with [], Dw y Dq} and \eqref{eq P_n+m + sum^M a P_n-i+m = S_n+m + sum^K c S_n-i+m, Dw y Dq}, \eqref{eq c_j,n,lambda,nu <S_n-j+m,S_n-j+m>_lambda,nu = sum^M + sum^N, Dw y Dq} becomes, for $n\geq j$ and $0\leq j\leq K$,
 \begin{align*}
  s_{n-j+m,\nu} c_{j,n,\lambda,\nu}
  & = \sum_{i=j}^M \sum_{\ell=0}^M \widetilde{a}_{i,n} \widetilde{a}_{\ell,n-j} \left\langle \U, P_{n-i+m}(x)P_{n-j-\ell+m}(x) \right\rangle \\
  & - \sum_{i=j}^M \sum_{\ell=1}^K \widetilde{a}_{i,n} c_{\ell,n-j,\lambda,\nu} \left\langle \U, P_{n-i+m}(x)S_{n-j-\ell+m}(x;\lambda,\nu) \right\rangle \\
  & + \lambda \sum_{i=j}^N \sum_{\ell=0}^N \widetilde{b}_{i,n} \widetilde{b}_{\ell,n-j} \left\langle \V, Q_{n-i}(x)Q_{n-j-\ell}(x) \right\rangle \\
  & - \lambda \sum_{i=j}^N \sum_{\ell=1}^{K} \widetilde{b}_{i,n} c_{\ell,n-j,\lambda,\nu} \left\langle \V, Q_{n-i}(x)D_\nu^m[S_{n-j-\ell+m}(x;\lambda,\nu)] \right\rangle.
 \end{align*}
 Since $\langle \U, P_{n-i+m}(x)S_{n-j-\ell+m}(x;\lambda,\nu)\rangle=0$ and $\langle \V, Q_{n-i}(x)D_\nu^m[S_{n-j-\ell+m}(x;\lambda,\nu)]\rangle=0$, for $i<j+\ell$ or $j+\ell>K\,(\geq M,N)$, thus, for $n\geq j$ and $0\leq j\leq K$,
 \begin{multline*}
  s_{n-j+m,\nu} c_{j,n,\lambda,\nu} = \sum_{i=j}^M \widetilde{a}_{i,n} \widetilde{a}_{i-j,n-j} \langle\U,P_{n-i+m}^2(x)\rangle + \lambda \sum_{i=j}^N \widetilde{b}_{i,n} \widetilde{b}_{i-j,n-j} \langle\V,Q_{n-i}^2(x)\rangle \\
  - \sum_{\ell=1}^{K-j} c_{\ell,n-j,\lambda,\nu} \sum_{i=j+\ell}^{M} \widetilde{a}_{i,n}\langle \U, P_{n-i+m}(x)S_{n-j-\ell+m}(x;\lambda,\nu)\rangle \\
  - \lambda \sum_{\ell=1}^{K-j} c_{\ell,n-j,\lambda,\nu} \sum_{i=j+\ell}^{N} \widetilde{b}_{i,n} \langle \V, Q_{n-i}(x)D_\nu^m[S_{n-j-\ell+m}(x;\lambda,\nu)] \rangle,
 \end{multline*}
 and from \eqref{eq c_j,n,lambda,nu <S_n-j+m,S_n-j+m>_lambda,nu = sum^M + sum^N, Dw y Dq}, the sum of the last two terms is $-\sum_{\ell=1}^{K-j} c_{\ell,n-j,\lambda,\nu} s_{n-j-\ell+m,\nu}$ $c_{j+\ell,n,\lambda,\nu}$. Finally, substituting $n$ by $n+j$, we get \eqref{eq s*c = zeta_j,n,lambda,nu - sum^K-j c*c*s, Dw y Dq}.
\end{proof}

\begin{remark} Notice that Theorem \ref{teo (M,N)-Dnu-coherence order (m,0) impl relacion con Sobolev, Dw y Dq}
allows recursively compute the $D_\nu$-Sobolev SMOP $\{S_n(x;\lambda,\nu)\}_{n\geq0}$ and the coefficients
$\{c_{j,n,\lambda,\nu}\}_{n\geq0}$, $1\leq j\leq K$.
 Moreover, Theorem \ref{teo (M,N)-Dnu-coherence order (m,0) impl eq recursiva s_n,nu y c_j,n,lambda,nu, Dw y Dq} gives a recursive equation for computing the sequences $\{c_{j,n,\lambda,\nu}\}_{n\geq0}$, $1\leq j\leq K$, and $\{\langle S_n(x;\lambda,\nu), S_n(x;\lambda,\nu)\rangle_{\lambda,\nu}\}_{n\geq0}$, and thus, using \eqref{eq P_n+m + sum^M a P_n-i+m = S_n+m + sum^K c S_n-i+m, Dw y Dq} and $S_{n}(x;\lambda,\nu)=P_n(x)$ for $n< m$, we can get the $D_\nu$-Sobolev SMOP $\{S_n(x;\lambda,\nu)\}_{n\geq0}$.
\end{remark}

\section*{Acknowledgements}
We are grateful to Prof. Francisco Marcell\'an for his valuable comments and remarks that helped us to
improve the paper.
This work was supported by Direcci\'{o}n General de Investigaci\'{o}n, Desarrollo e Innovaci\'{o}n, Ministerio de Econom\'{i}a y Competitividad of Spain, under grants MTM2012-36732-C03 (RAN, NCP-C, JP),
Junta de Andaluc\'{\i}a (Spain) under grants FQM262, FQM-7276, and P09-FQM-4643 (RAN), FEDER funds (RAN). The work of J. Petronilho was also partially supported by the Centro de Matem\'atica da Universidade de Coimbra (CMUC), funded by the European Regional Development Fund through the program COMPETE and by the Portuguese Government through the FCT - Funda\c{c}\~ao para a Ci\^encia e a Tecnologia under the project PEst-C/MAT/UI0324/2013.
The work of R. Sevinik was supported by T\"{U}B\.{I}TAK, the Scientific and
Technological Research Council of Turkey.

\end{document}